\documentclass[a4paper]{amsart}
\usepackage[utf8]{inputenc}
\usepackage{graphicx,tikz}
\usetikzlibrary{shapes}
\usepackage{tikz-cd}
\usetikzlibrary{calc}
\usepackage{amssymb}

\theoremstyle{plain}
  \newtheorem{thm}{Theorem}
  \newtheorem{defn}{Definition}
  \newtheorem{prop}{Proposition}

\theoremstyle{definition}
  \newtheorem{example}{Example}
\theoremstyle{remark}
  \newtheorem{rem}{Remark}

\newcommand{\on}{\operatorname}
\newcommand{\g}{\mathfrak{g}}
\newcommand{\h}{\mathfrak{h}}

\newcommand{\la}{\langle}
\newcommand{\ra}{\rangle}
\newcommand{\R}{\mathbb{R}}

\usetikzlibrary{matrix,arrows,decorations.pathmorphing,decorations.markings,calc}
\tikzset{->-/.style={decoration={
  markings,
  mark=at position #1 with {\arrow{stealth'}}},postaction={decorate}}}
\tikzset{proj/.style={circle,fill=black!20,draw=black!20,inner sep=0pt}}
\tikzset{anch/.style={circle,fill=black,draw=black,inner sep=0pt,minimum size=1.5mm}}

\title[{Ricci flow, Courant algebroids, and Poisson--Lie T-duality}]{Ricci flow, Courant algebroids, and renormalization of Poisson--Lie T-duality}
\thanks{Supported in part by  the grant MODFLAT of the European Research Council and the NCCR SwissMAP of the Swiss National Science Foundation.}
\author{Pavol \v{S}evera}
\address{Section of Mathematics, University of Geneva, Switzerland}
\email{pavol.severa@gmail.com}
\author{Fridrich Valach}
\address{Section of Mathematics, University of Geneva, Switzerland}
\email{fridrich.valach@gmail.com}

\begin{document}
\maketitle
\begin{abstract}
We use a generalized Ricci tensor, defined for generalized metrics in Courant algebroids, to show that Poisson-Lie T-duality is compatible with the 1-loop renormalization group.
\end{abstract}
\section{Introduction}
 If $(M,g)$ is a Riemannian manifold, the 1-loop renormalization group flow of the standard 2-dimensional $\sigma$-model with the action functional
$$S(f)=\int_\Sigma g(\partial f,\bar\partial f)\qquad(f:\Sigma\to M)$$
is (up to an inessential coefficient), as found by Friedan \cite{f},   the Ricci flow
$$\frac{dg}{dt}=-2\on{Ric}_g+L_Xg,$$
where $X$ is an arbitrary vector field on $M$.

In a more general situation $\on{Ric}_g$ gets replaced by $\on{Ric}_{(g,H)}$ where $H$ is a closed 3-form and $\on{Ric}_{(g,H)}$ is the Ricci tensor of the $g$-preserving connection on $M$ with the torsion $T$ given by $g(T(X,Y),Z)=-H(X,Y,Z)$.
Namely, if $b\in\Omega^2(M)$ is a 2-form and $H_0\in\Omega^3(M)$ a closed 3-form,  the corresponding $\sigma$-model has the action functional 
$$S(f)=\int_\Sigma g(\partial f,\bar\partial f)+\int_\Sigma f^*b+\int_Yf^*H_0=\int_\Sigma e(\partial f,\bar\partial f)+\int_Yf^*H_0$$
where  $e=g+b\in\Gamma((T^*)^{\otimes 2}M)$ and $Y$ is an oriented 3-manifold bounded by $\Sigma$.  We set $H=db+H_0$ and the 1-loop renormalization group flow is now \cite{fv}
\begin{equation}\label{rflow}
\frac{de}{dt}=-2\on{Ric}_{(g,H)}+L_X e+i_XH_0-d\alpha
\end{equation}
where $X$ and $\alpha$ are an arbitrary vector field and a 1-form respectively.

It is natural to interpret the generalized Ricci flow \eqref{rflow} in terms of Courant algebroids. The pair $(X,\alpha)$ can be seen as a section of an exact Courant algebroid $(T\oplus T^*)M$ and its contribution to the flow as the Courant bracket $[(X,\alpha),\cdot]$. The tensor field $e$ can be replaced by its graph which is a subbundle of $(T\oplus T^*)M$ and \eqref{rflow} can be seen as an evolution of this subbundle. This point of view  was suggested by Streets in \cite{str}.

The main problem in this approach is to find a version of the flow \eqref{rflow} for arbitrary Courant algebroids, not just for $(T\oplus T^*)M$ (i.e.\ exact ones). The motivation for this problem comes mainly from T-duality. T-duality \cite{b}, and its extension Poisson-Lie T-duality \cite{ks}, is an equivalence of $\sigma$-models with different target spaces $M$. It has a natural formulation in terms of Courant algebroids (see \cite{cg} for the ordinary and \cite{s-ca} for the Poisson-Lie case). A~suitable definition of a generalized Ricci flow for an arbitrary Courant algebroid would immediately imply the compatibility of T-duality with 1-loop renormalization group flow. In this paper we give such a definition and prove that Poisson-Lie T-duality is indeed compatible with the 1-loop renormalization group flow. This result was known only in the case of no spectators \cite{vks} and the new proof is much simpler.

We should mention that this is not the first attempt to define the Ricci tensor of a generalized metric. A previous, and more conceptual way, was to use Gualtieri's definition of a generalized connection and its generalized curvature \cite{gbranes}, choose the connection so that it preserves the generalized metric and a suitable part of its generalized torsion vanishes, and take a suitable trace of the curvature. This approach was used by Garcia-Fernandez \cite{gf} to get a new interpretation  of the equations of motion of heterotic supergravity. Our approach is more utilitarian.

\subsection*{Acknowledgments}
We would like to thank to Marco Gualtieri for comments on a preliminary version of this paper, in particular for telling us about generalized connections and their curvature \cite{gbranes} and about the previous work on generalized Ricci tensor \cite{gf}.

\section{Courant algebroids}
In this section we shall summarize basic definitions and results concerning Courant algebroids.
Courant algebroids were introduced by Liu, Weinstein and Xu in \cite{lwx}.
\begin{defn}
A \emph{Courant algebroid (CA)} is a vector bundle $E\to M$ equipped with a non-degenerate symmetric bilinear form $\la\,,\,\ra$, with a vector bundle map 
$$\rho:E\to TM$$ (the \emph{anchor map}) and with a $\R$-bilinear map (\emph{Courant bracket})
$$[\,,\,]:\Gamma(E)\times\Gamma(E)\to\Gamma(E)$$
satisfying
\begin{itemize}
\item $[s,[t,u]]=[[s,t],u]+[t,[s,u]]$ for any $s,t,u\in\Gamma(E)$ \label{ax1}
\item $\rho([s,t])=[\rho(s),\rho(t)]$ for any $s,t\in\Gamma(E)$ \label{ax2}
\item $[s,ft]=f[s,t]+(\rho(s)f)t$ for any $s,t\in\Gamma(E)$, $f\in C^\infty(M)$ \label{ax3}
\item $\rho(s)\la t,u\ra=\la [s,t],u\ra+\la t,[s,u]\ra$ \label{ax4}
\item $[s,t]+[t,s]=\rho^t\bigl(d\la s,t\ra\bigr)$, where \label{ax5}
$$\rho^t:T^*M\to E^*\xrightarrow{\la,\ra} E$$
is the transpose of $\rho$.
\end{itemize}
\end{defn}
One can reformulate the first four properties as follows: if $s\in\Gamma(E)$ then the map $\Gamma(E)\to\Gamma(E)$, $x\mapsto[s,x]$, is a derivation of the Courant algebroid $E$, i.e.\ it is given by a vector field $Z_s$ on $E$  whose flow is an automorphism of the CA $E$,  such that $\rho_*Z_s=\rho(s)$. We shall call the map $x\mapsto[s,x]$ an \emph{inner derivation} of $E$.
 
\begin{example}
If $M$ is a point then $E$ is a Lie algebra with invariant non-degenerate quadratic form $\la\,,\,\ra$.
\end{example}

\begin{example}
A Courant algebroid $E\to M$ is called \emph{exact} if 
\begin{equation}\label{exact-seq}
0\to T^*M\xrightarrow{\rho^t} E\xrightarrow{\rho} TM\to 0
\end{equation}
 is an exact sequence (it is a chain complex for any Courant algebroid). 
 Exact CAs over $M$ are classified by $H^3(M,\R)$. Namely, if we choose a splitting of \eqref{exact-seq} by a $\la,\ra$-isotropic subbundle of $E$, we get $E\cong(T\oplus T^*)M$ with
\begin{subequations}\label{exact-CA}
\begin{align}
\la(X,\alpha),(Y,\beta)\ra&=\alpha(Y)+\beta(X),\\
\rho(X,\alpha)&=X,\\
[(X,\alpha),(Y,\beta)]&=([X,Y],L_X\beta-i_Y d\alpha+H(X,Y,\cdot))\label{exact-CA-br}
\end{align}
\end{subequations}
for some closed 3-form $H\in\Omega^3(M)$, and vice-versa, any closed $H$ makes $(T\oplus T^*)M$ in this way to an exact CA. A different choice of the splitting replaces $H$ by $H+dB$ for an appropriate $B\in\Omega^2(M)$.
\end{example}

It is convenient to write the Courant bracket $[,]$ on $E$ in terms of a connection on $E$:
\begin{prop}\label{prop:c}
Let $E\to M$ be a CA and let $\nabla$ be a connection on the vector bundle $E$ preserving the pairing $\la,\ra$. Then there is a section $c_\nabla\in\Gamma(\bigwedge^3 E)$ such that for every $u,v\in\Gamma(E)$ we have
\begin{equation}\label{c-nabla}
[u,v]=c_\nabla(u,v,\cdot)+\nabla_{\rho(u)}v-\nabla_{\rho(v)}u+\rho^t\la\nabla u,v\ra
\end{equation}
where we identify $E^*$ with $E$ via the pairing $\la,\ra$.
\end{prop}
\begin{proof}
Let us set
$$[u,v]_\nabla:=[u,v]-\bigl(\nabla_{\rho(u)}v-\nabla_{\rho(v)}u+\rho^t\la\nabla u,v\ra\bigr).$$
One easily sees that $[u,v]_\nabla$ is $C^\infty(M)$-bilinear in $u$ and $v$, i.e.\ it is a vector bundle map $E\otimes E\to E$.

 It remains to check that
$$c_\nabla(u_1,u_2,u_3)=\la[u_1,u_2]_\nabla,u_3\ra$$
is antisymmetric in its 3 arguments $u_1,u_2,u_3\in\Gamma(E)$. To do it, we choose $u_i$'s such that $\nabla u_i$'s vanish at a point $P\in M$. In that case $\la[u_1,u_2]_\nabla,u_3\ra(P)=\la[u_1,u_2],u_3\ra(P)$ and $\la[u_1,u_2],u_3\ra(P)$ is antisymmetric as $d\la u_i,u_j\ra$'s vanish at $P$.
\end{proof}

\begin{rem}
In \cite{gbranes} the section $c_\nabla$ is called the torsion of $\nabla$ and is defined for generalized connections on $E$.
\end{rem}

\section{Generalized metric}
A \emph{generalized metric} on a Courant algebroid $E\to M$, as defined in \cite{g2}, is  a vector subbundle $V_+\subset E$ such that  $\la,\ra$ is positive definite on $V_+$ and negative definite on $V_-:=V_+^\perp$.

If $E\to M$ is an exact CA with a chosen splitting, i.e.\ if $E=(T\oplus T^*)M$ with the structure given by \eqref{exact-CA}, then a generalized metric $V_+\subset E$ is the graph of a bilinear form $e=g+b$ on $TM$ such that the symmetric part $g$ of $e$ is a Riemannian metric on $M$. The skew-symmetric part $b$ of $e$ depends on the splitting ($g$ does not), and for a given $V_+$ there is a unique splitting of $E$ such that\ $b=0$ (see \cite{g2}); the closed 3-form corresponding to this splitting will be denoted $H$. An exact CA with a generalized metric is thus equivalent to a pair $(g,H)$.

Let now $\nabla^\pm$ be connections on the vector bundles $V_\pm$ preserving $\la,\ra$, and let $\nabla=\nabla^+\oplus\nabla^-$ be the resulting connection on $E=V_+\oplus V_-$. We shall call such a connection $\nabla$ \emph{compatible} with the generalized metric $V_+\subset E$. In the case of an exact CA there is a canonical connection $\nabla^-$ for which the ``$+--$''-part  of $c_\nabla\in\Gamma(\bigwedge^3 E)$ vanishes. This result can be found in \cite{gbranes}, but we include a proof for completeness.

\begin{prop}\label{prop:nabla-can}
If $E\to M$ is an exact CA and $V_+\subset E$ a generalized metric, there is a unique $\la,\ra$-preserving connection $\nabla^-_\mathrm{can}$ on $V_-$ such that, for any $\la,\ra$-preserving $\nabla^+$ on $V_+$, we have 
\begin{equation}\label{nab-aux}
c_{\nabla^+\oplus\nabla^-_\mathrm{can}}(x_+,y_-,z_-)=0
\end{equation}
 for any $x_+\in\Gamma(V_+)$ and $y_-,z_-\in\Gamma(V_-)$. When we identify $V_-$ with $TM$ via the anchor $\rho$ then $\nabla^-_\mathrm{can}$ is the the $g$-preserving connection with the torsion $T$ given by $g(T(X,Y),Z)=-H(X,Y,Z)$.
\end{prop}
\begin{proof}
Let $\nabla=\nabla^+\oplus\nabla^-$ be a compatible connection on $E$.
The relation \eqref{c-nabla} gives
\begin{equation}\label{c-vs-nabla}
\la[x_+,y_-],z_-\ra=c_\nabla(x_+,y_-,z_-)+\la\nabla_{\rho(x_+)}y_-,z_-\ra.
\end{equation}
As a result, if we change $\nabla$ by a 1-form 
$${a}={a}_++{a}_-,\ {a}_\pm\in\Omega^1(M,\textstyle\bigwedge^2V_\pm),$$
 we get
\begin{equation}\label{delta-c}
(c_{\nabla+{a}}-c_\nabla)(x_+,y_-,z_-)=-{a}_-(\rho(x_+))(y_-,z_-).
\end{equation}
For a given $\nabla=\nabla^+\oplus\nabla^-$ we then define ${a}_-\in\Omega^1(M,\bigwedge^2V_-)$ via
$${a}_-(\rho(x_+))(y_-,z_-)=c_\nabla(x_+,y_-,z_-)$$
(we use the fact that $\rho:V_+\to TM$ is bijective) and see that the connection $\nabla^-_\mathrm{can}=\nabla^-+{a}_-$ on $V_-$ is the unique solution of our problem.

Let us now compute the torsion of $\nabla^-_\mathrm{can}$ using \eqref{nab-aux}, i.e.\ using (cf. \eqref{c-vs-nabla})
\begin{equation}\label{xyz-proof}
\la[x_+,y_-],z_-\ra=\la(\nabla^-_{\mathrm{can}})_{\rho(x_+)}y_-,z_-\ra.
\end{equation}
We split $E$ to $(T\oplus T^*)M$ via the splitting for which $V_+$ is the graph of the Riemann metric $g$, and thus $V_-$ is the graph of $-g$. Notice that for $v_-,w_-\in\Gamma(V_-)$ we have
\begin{equation}\label{lara-g}
\la v_-,w_-\ra=-2g(\rho(v),\rho(w)).
\end{equation}

Let $X$ and $Y$ be vector fields on $M$ such that $[X,Y]=0$, and let
$$
x_\pm=(X,\pm g(X,\cdot)),\ y_\pm=(Y,\pm g(Y,\cdot))\in\Gamma(V_\pm)
$$
be their lifts to $V_\pm$. Let $z_-=(Z,-g(Z,\cdot))$ be a section of $V_-$. Then \eqref{exact-CA-br} gives us
$$[x_+,y_-]-[y_+,x_-]=(0,2H(X,Y,\cdot))$$
and \eqref{xyz-proof} and \eqref{lara-g} give us
$$-2g(T(X,Y),Z)=\la[x_+,y_-]-[y_+,x_-],z_-\ra=2H(X,Y,Z)$$
as we wanted to show.
\end{proof}

Let us now introduce the following graphical notation. We will systematically identify $E^*$ with $E$ via the pairing $\la,\ra$. The section $c_\nabla\in\Gamma(\bigwedge^3E)$ will be represented by a trivalent vertex
$$c_\nabla=
\begin{tikzpicture}[baseline=-0.6ex, scale=0.5,thick]
\draw(0,0)--(-1,0) (0,0)--(0.5,0.866) (0,0)--(0.5,-0.866);
\end{tikzpicture}
$$
with the counter-clockwise orientation, i.e.\ for any $u,v,w\in\Gamma(E)$
$$c_\nabla(u,v,w)=
\begin{tikzpicture}[baseline=-0.6ex, scale=0.6,thick]
\node (u) at (-1,0) {$u$};
\node (v) at (0.5,-0.866) {$v$};
\node (w) at (0.5,0.866) {$w$};
\draw(u)--(0,0) (v)--(0,0) (w)--(0,0);
\end{tikzpicture}
$$
If $V_+\subset E$ is a generalized metric,
let the orthogonal projections $E\to V_\pm$ be denoted by \tikz[baseline=-0.6ex]\draw[thick](-0.5,0)--(0,0) node[proj]{\tiny$+$}--(0.5,0); and \tikz[baseline=-0.6ex]\draw[thick](-0.5,0)--(0,0) node[proj]{\tiny$-$}--(0.5,0); respectively. In particular, if $E$ is exact and if  $\nabla^-=\nabla^-_\mathrm{can}$ then
$$
\begin{tikzpicture}[baseline=-0.6ex, scale=0.6,thick, scale=0.6]
\node (1) at (-1,0)[proj] {\tiny$+$};
\node (2) at (0.5,-0.866)[proj] {\tiny$-$};
\node (3) at (0.5,0.866)[proj] {\tiny$-$};
\draw(0,0)--(1)--(-2,0) (0,0)--(2)--(1,-1.73) (0,0)--(3)--(1,1.73);
\end{tikzpicture}
=0.
$$
Finally, let us also use the notation
$$\rho=
\begin{tikzpicture}[baseline=-0.6ex,thick, scale=0.6]
\node (rho) at (0,0) [anch]  {};
\draw (-1,0)--(rho);
\draw[dotted,->-=0.6] (rho)--(1.5,0);
\end{tikzpicture}
\qquad
\text{and}
\qquad
\nabla c=
\begin{tikzpicture}[baseline=-0.6ex, thick, scale=0.5,thick]
\draw(0,0)--(1,0) (0,0)--(-0.5,0.866) (0,0)--(-0.5,-0.866);
\draw[dotted,->,>=open triangle 60] (-1.5,0)--(0,0);
\end{tikzpicture}
$$
where dotted lines signify vector fields on $M$; more generally, \tikz{\draw[thick,dotted,->,>=open triangle 60] (-1,0)--(0,0);} will stand for $\nabla$ (always applied to some section of $E^{\otimes n}$).

\section{Generalized Ricci tensor}

Let $V_+\subset E$ be a generalized metric. An infinitesimal deformation of $V_+$, i.e.\ a tangent vector to $V_+$ in the space of generalized metrics in $E$, is given by a linear map $S:V_+\to E/V_+\cong V_-$, or equivalently by the corresponding bilinear form 
$$C:V_+\otimes V_-\to\R,\quad C(u_+,v_-)=\la S u_+, v_-\ra.$$ 

Deformations by inner derivations $[s,\cdot]$ of $E$ have the bilinear form 
$$C(u_+,v_-)=\la [s,u_+], v_-\ra.$$
 These deformations are trivial in the sense that they don't change the isomorphism class of the pair $V_+\subset E$.

We can now define the main object of this paper.
\begin{defn}
The \emph{generalized Ricci tensor} of a generalized metric $V_+$ in a Courant algebroid $E$ with a compatible connection $\nabla=\nabla^+\oplus\nabla^-$ is the bilinear form $\on{GRic}_{V_+}^{(\nabla)}:V_+\otimes V_-\to\R$ given by, for $u_+\in\Gamma(V_+)$ and $v_-\in\Gamma(V_-)$,
\begin{multline}\label{GRic}
\on{GRic}_{V_+}^{(\nabla)}(u_+,v_-):=
\on{Tr}_{V_-}\Bigl(x_-\mapsto R_{\nabla^-}\bigl(\rho(x_-),\rho(u_+)\bigr)v_-\Bigr)\\
-
\begin{tikzpicture}[scale=0.6, baseline=-0.6ex, thick]
\node (u) at (-0.5,0) {$\strut u_+$};
\coordinate (x) at (1,0);
\coordinate (y) at (3,0) {};
\node (v) at (4.5,0) {$\strut v_-$};
\node (plus) at (2,0.8) [proj] {\tiny $+$};
\node (minus) at (2,-0.8) [proj] {\tiny$-$};
\draw (u)--(x) (y)--(v);
\draw (x) to[bend left] (plus) (plus)to[bend left] (y)
      (x) to[bend right] (minus) (minus)to[bend right] (y) ;
\end{tikzpicture}
+
\begin{tikzpicture}[scale=0.6, baseline=-0.6ex, thick]
\node (u) at (-1,-1) {$u_+$};
\node (v) at (1,-1) {$v_-$};
\coordinate (x) at (0,0);
\node (minus) at (-0.2,1) [proj] {\tiny$-$};
\draw (u)--(x)--(v) (x)to[out=80, in=-60](minus);
\node (rho) at (-1,0.7) [anch]  {};
\draw (minus) to[out=140,in=90] (rho);
\draw[dotted,->,>=open triangle 60] (rho) to[out=-70,in=160] (x);
\end{tikzpicture}
\end{multline}
where $R_{\nabla^-}$ is the curvature of $\nabla^-$.
\end{defn}

\begin{rem}
In this definition one can replace $\nabla$ with a generalized connection in the sense of \cite{gbranes}. Moreover, the result is probably equal, up to inner derivations, to the generalized Ricci tensor of a torsion-free generalized connection  \cite{gf}. For our purposes ordinary connections are sufficient.
\end{rem}

The infinitesimal deformation of $V_+$ given by $\on{GRic}_{V_+}^{(\nabla)}$ is independent of $\nabla$ modulo inner derivations (and is actually fully independent of $\nabla^+$):

\begin{thm}\label{thm:GRic}
Let $V_+\subset E$ be a generalized metric in a Courant algebroid $E$ and let $\nabla=\nabla^+\oplus\nabla^-$ be a compatible connection. If $\nabla+{a}$, ${a}={a}_++{a}_-$, $a_\pm\in\Omega^1(M,\bigwedge^2 V_\pm)$, is another compatible connection then
$$\on{GRic}_{V_+}^{(\nabla+{a})}(u_+,v_-)-\on{GRic}_{V_+}^{(\nabla)}(u_+,v_-)=\la [s_-,u_+],v_-\ra$$
where $s_-\in\Gamma(V_-)$ is 
\begin{equation}\label{s-minus}
s_-=
\begin{tikzpicture}[baseline=-0.6ex, thick, scale=0.8]
\coordinate (x) at (0,-1);
\coordinate[label=-10:${a}_-$] (alpha) at (0,0);
\node[anch] (rho) at (0,1) {};
\draw (x)--(alpha)..controls +(1,0.5) and +(0.5,0).. (rho);
\draw[dotted, ->-=0.75] (rho) ..controls +(-0.5,0) and +(-1,0.5).. (alpha);
\end{tikzpicture}
\end{equation}
Here we use the graphical notation
$${a}_\pm(X)(x,y)=
\begin{tikzpicture}[scale=1,baseline=-0.6ex, thick]
\node (X) at (-1.2,0) {$X$};
\node (s) at (0.5,-0.866) {$x$};
\node (t) at (0.5,0.866) {$y$};
\coordinate (alpha) at (0,0);
\draw (alpha) node[right] {$\strut{a}_\pm$};
\draw[dotted, ->-=0.6] (X)--(alpha);
\draw (s)--(alpha) (t)--(alpha);
\end{tikzpicture}
$$
for a vector field $X$ and sections $x,y\in\Gamma(E)$.
\end{thm}
The proof of Theorem \ref{thm:GRic} is a straightforward calculation and can be found in Appendix \ref{appendix}.

Let us recall that a generalized metric in an exact CA is equivalent to a Riemannian metric $g$ and a closed 3-form $H$. Let $\on{Ric}_{(g,H)}$ be the Ricci tensor of the $g$-preserving connection with the torsion given by $g(T(X,Y),Z)=-H(X,Y,Z)$ (cf.\ Proposition \ref{prop:nabla-can}).
\begin{thm}\label{thm:GRic-ex}
If $V_+\subset E$ is a generalized metric in an exact CA $E$ and $\nabla=\nabla^+\oplus\nabla^-$ a compatible connection then
\begin{equation}\label{GRic-exact}
\on{GRic}_{V_+}^{(\nabla)}(u_+,v_-)=\on{Ric}_{(g,H)}(\rho(u_+),\rho(v_-))+\la [s_-,u_+],v_-\ra
\end{equation}
where $s_-\in\Gamma(V_-)$ is given by \eqref{s-minus} with ${a}_-=\nabla^--\nabla^-_\mathrm{can}$. 
\end{thm}
\begin{proof}
If $\nabla^-=\nabla^-_\mathrm{can}$ then 
$$
\begin{tikzpicture}[baseline=-0.6ex, thick, scale=0.4]
\node (1) at (-1,0)[proj] {\tiny$+$};
\node (2) at (0.5,-0.866)[proj] {\tiny$-$};
\node (3) at (0.5,0.866)[proj] {\tiny$-$};
\draw(0,0)--(1)--(-2,0) (0,0)--(2)--(1,-1.73) (0,0)--(3)--(1,1.73);
\end{tikzpicture}
=0
$$
and so the second and third term of \eqref{GRic} vanish and we have 
$$\on{GRic}_{V_+}^{(\nabla)}(u_+,v_-)=\on{Tr}_{V_-}\Bigl(x_-\mapsto R_{\nabla^-_\text{can}}\bigl(\rho(x_-),\rho(u_+)\bigr)v_-\Bigr)=\on{Ric}_{(g,H)}(\rho(u_+),\rho(v_-)).$$
For any compatible $\nabla$ we therefore have
$$\on{GRic}_{V_+}^{(\nabla-{a}_-)}(u_+,v_-)=\on{Ric}_{(g,H)}(\rho(u_+),\rho(v_-))$$
where ${a}_-=\nabla^--\nabla^-_\mathrm{can}$. Equation \eqref{GRic-exact} now follows from Theorem \ref{thm:GRic}.
\end{proof}

Theorem \ref{thm:GRic-ex} has the following meaning. If $E=(T\oplus T^*)M$ is exact given by a closed 3-form $H_0$ and $V_+\subset E$ is the graph of $e=g+b$, then an infinitesimal deformation of $V_+$ can be given either by a bilinear form $C:V_+\otimes V_-\to\R$, or by a deformation $\frac{de}{dt}$ of $e$; they are linked via 
$$\frac{de}{dt}(\rho(u_+),\rho(v_-))=C(u_+,v_-).$$
 The deformation of $V_+$ via $-2 \on{GRic}_{V_+}^{(\nabla)}$ is thus precisely the generalized Ricci flow \eqref{rflow} with $(X,\alpha)=-2s_-$ (notice that $\alpha=-e(\cdot, X)$ as $s_-\in\Gamma(V_-)$).

\begin{rem}
Theorems \ref{thm:GRic} and \ref{thm:GRic-ex} explain why $\on{GRic}_{V_+}^{(\nabla)}$ is useful, but its definition is somewhat ad hoc. A natural guess is that $\on{GRic}_{V_+}^{(\nabla)}$ comes from one-loop renormalization of the Courant $\sigma$-model given by $E$ with the boundary condition given by $V_+$, introduced in \cite{s-cs}.
\end{rem}

\begin{rem}
We defined $\on{GRic}^{(\nabla)}_{V_+}$ as an infinitesimal deformation of $V_+$. An obvious question is whether it defines, say in the case of a compact $M$, a well-defined flow for some finite time. In other words, whether one can find $t_0>0$, a family $V_+(t)$ parametrized by $t\in[0,t_0)$ such that $V_+(0)=V_+$, and a family of connections $\nabla(t)$ compatible with $V_+(t)$, and possibly a family of sections $z_-(t)\in\Gamma(V_-(t))$, such that 
$$\frac{dV_+(t)}{dt}=-2\on{GRic}^{(\nabla(t))}_{V_+(t)}+\la[z_-(t),\cdot],\cdot\ra$$
(where, as before, we identify tangent vectors to $V_+$ in the space of generalized metrics in $E$ with bilinear forms $V_+\otimes V_-\to\R$).
The section $z_-(t)\in\Gamma(V_-(t))$ is added because of the possibly arbitrary family of connections $\nabla(t)$ (cf.\ Theorem \ref{thm:GRic}). We leave this question open.
\end{rem}

\begin{rem}
Another question is, supposing the flow exists, how to describe its outcome in a way that would be independent of $\nabla(t)$ and $z_-(t)$, i.e.\ how to deal with the fact that $\on{GRic}^{(\nabla)}_{V_+}$ depends on the auxiliary connection $\nabla$, but only by inner derivations by sections of $V_-$. Here the answer is simple. The outcome should be a submersion $p:\hat M\to[0,t_0)$, a Courant algebroid $\hat E\to\hat M$ such that $dp\circ\rho$ is surjective at all points of $\hat M$, and a vector subbundle $\hat V_+\subset\hat E$ such that   $\la,\ra|_{\hat V_+}$ is positive definite and $dp\circ\rho|_{\hat V_+}=0$. 

If we set $M=p^{-1}(0)$, we can find a CA $E\to M$, a (at least local, and global if $p$ is proper) diffeomorphism $\hat M\cong M\times[0,t_0)$ compatible with $p$, and an isomorphism of CAs $\hat E\cong E\times(T\oplus T^*)[0,t_0)$ such that $\hat V_+$ becomes a family $V_+(t)$ of subbundles of $E$ (these isomorphisms can be obtained by choosing a section $x\in\Gamma(\hat V_+^\perp)$ such that $p_*(\rho(x))=\partial/\partial t$ and $\la x,x\ra=0$ and using the flow in $\hat E$ generated by the inner derivation $[x,\cdot]$). The resulting family $V_+(t)$ depends on the choices, and it changes by the flow of a time-dependent inner derivation generated by a section of $V_-(t)$ if we make a different choice.
\end{rem}

\section{Poisson-Lie T-duality is compatible with the 1-loop renormalization group flow}

\begin{thm}\label{thm:pltd}
Let $E\to M$ be a CA, $\phi:M'\to M$ a smooth map, and let $\phi^*E$ be endowed with a CA structure satisfying the condition
\begin{equation}\label{pullback-CA}
[\phi^*u,\phi^*v]=\phi^*[u,v]\text{ and }\phi_*\bigl(\rho(\phi^*u)\bigr)=\rho(u)\quad\forall u,v\in\Gamma(E).
\end{equation}
 Let $V_+\subset E$ be a generalized metric and $\nabla=\nabla^+\oplus\nabla^-$ a compatible connection on $E$. Then the generalized Ricci tensors of $V_+\subset E$ and of $\phi^*V_+\subset \phi^*E$ satisfy 
$$\on{GRic}^{(\phi^*\nabla)}_{\phi^*V_+}=\phi^*\on{GRic}_{V_+}^{(\nabla)}.$$
\end{thm}
\begin{proof}
If $u,v,w\in\Gamma(E)$ then $\la[u,v],w\ra=\la[\phi^*u,\phi^*v],\phi^*w\ra$. When we express both sides of this equality using Proposition \ref{prop:c}, we get $c_{\phi^*\nabla}=\phi^*c_\nabla$. From this and from the definition of $\on{GRic}$ the statement follows readily.
\end{proof}

\begin{rem}
If $E\to M$ is a CA and $\phi:M'\to M$ a smooth map, the CA structures on $\phi^*E$ satisfying the condition \eqref{pullback-CA} were characterized by Li-Bland and Meinrenken \cite{lbm} as follows: if $\rho:\phi^*E\to TM'$ is a vector bundle map then such a CA structure on $\phi^*E$ with the anchor $\rho$ exists (and moreover is unique) iff
\begin{itemize}
\item $\phi_*\bigl(\rho(\phi^*u)\bigr)=\rho(u)\quad\forall u\in\Gamma(E)$
\item $[\rho(\phi^*u),\rho(\phi^*v)]=\rho(\phi^*[u,v])\quad\forall u,v\in\Gamma(E)$ 
\item for any $p\in M'$ the kernel of $\rho$ at $p$ is a coisotropic subspace of $E_{\phi(p)}$.
\end{itemize}
In particular, if $M$ is a point and hence $E=\g$ is a Lie algebra, a CA structure on $\g\times M'$, such that the Courant bracket of constant sections is the Lie bracket in $\g$, is equivalent to an action $\rho$ of $\g$ on $M'$ with coisotropic stabilizers. This CA is exact iff the action is transitive with Lagrangian stabilizers, i.e.\ if $M'$ is (locally diffeomorphic to) $G/H$ where $\h\subset\g$ satisfies $\h^\perp=\h$.
\end{rem}

Theorem \ref{thm:pltd} implies that Poisson-Lie T-duality is compatible with the 1-loop renormalization group flow. Let us  summarize the needed definitions. Suppose that $E\to M$ is a  CA, $\phi_i:M_i\to M$, $i=1,2$, are surjective submersions, and that we have  \emph{exact} CA structures on $\phi_i^*E$ satisfying \eqref{pullback-CA}.\footnote{If $G$ is a connected Lie group and $\la,\ra$ an invariant inner product on $\g$, then any principal $G$-bundle $P\to M$ with vanishing 1st Pontryagin class $[\la F,F\ra]\in H^4(M,\R)$ gives a transitive CA $E\to M$ (depending on a choice of a $\omega\in\Omega^3(M)/d\Omega^2(M)$ such that $d\omega=\la F,F\ra$) and for any Lie subgroup $H\subset G$ with $\h^\perp=\h$ we have a compatible exact CA structure on $\phi^*E\to P/H$ where $\phi:P/H\to M$ is the projection. This is the main source of examples. (In the case of $M=\text{point}$ we have $E=\g$ and so $\phi^* E=\g\times G/H$.) See \cite{s-ca} for details.} If $V_+\subset E$ is a generalized metric then $\phi_i^*V_+\subset \phi_i^*E$ are generalized metrics in the exact CAs $\phi_i^*E$ and these generalized metrics on $M_1$ and $M_2$ are said to be Poisson-Lie T-dual to each other.\footnote{It implies that the 2-dimensional $\sigma$-models with the target spaces $M_1$ and $M_2$ are (after a suitable reduction) isomorphic as Hamiltonian systems. See \cite{s-cs} for details.} 

If we choose a compatible connection $\nabla$ on $E$ and deform $V_+$ by $-2\on{GRic}_{V_+}^{(\nabla)}$ then by Theorem \ref{thm:pltd} the subbundle $\phi_i^*V_+\subset\phi^*E$ gets deformed by $-2\on{GRic}^{(\phi_i^*\nabla)}_{\phi_i^*V_+}$, i.e.\ by the 1-loop renormalization group flow
\eqref{rflow}. The deformed $\phi_i^*V_+$'s stay Poisson-Lie T-dual to each other, and so Poisson-Lie T-duality is indeed compatible with the 1-loop renormalization group flow.

\begin{rem}
The compatibility of Poisson-Lie T-duality with the 1-loop renormalization group flow was shown in the special case of $M=\text{point}$ (i.e.\ in the case of no ``spectators'') in \cite{vks} (under certain mild additional conditions). In this case we have $E=\g$ for some Lie algebra $\g$ with an invariant $\la,\ra$ and  $\on{GRic}$ simplifies to
$$
\on{GRic}_{V_+}(u_+,v_-)=
-
\begin{tikzpicture}[scale=0.6, baseline=-0.6ex, thick]
\node (u) at (-0.5,0) {$\strut u_+$};
\coordinate (x) at (1,0);
\coordinate (y) at (3,0) {};
\node (v) at (4.5,0) {$\strut v_-$};
\node (plus) at (2,0.8) [proj] {\tiny $+$};
\node (minus) at (2,-0.8) [proj] {\tiny$-$};
\draw (u)--(x) (y)--(v);
\draw (x) to[bend left] (plus) (plus)to[bend left] (y)
      (x) to[bend right] (minus) (minus)to[bend right] (y) ;
\end{tikzpicture}
$$
This expression was discovered in \cite{sst} as the 1-loop renormalization of a duality-invariant Hamiltonian version of the corresponding $\sigma$-models from \cite{ks2}. Proving that this expression really gives the generalized Ricci flow \eqref{rflow} and extending it to Poisson-Lie T-duality with spectators (i.e.\ with non-trivial $M$) was the original motivation of this paper.

\end{rem}

\appendix

\section{Proof of Theorem \ref{thm:GRic}}\label{appendix}
Let us compute
$$\delta_{a}\on{GRic}_{V_+}^{(\nabla)}:=\frac{d}{dt}\on{GRic}_{V_+}^{(\nabla+t{a})}\Big|_{t=0}.$$
The three terms of $\on{GRic}_{V_+}^{(\nabla)}$ contribute
\begin{multline*}
\delta_{a}\on{Tr}_{V_-}\Bigl(x_-\mapsto R_{\nabla^-}\bigl(\rho(x_-),\rho(u_+)\bigr)v_-\Bigr)\\
=
\begin{tikzpicture}[thick,scale=1,baseline=-0.6ex]
\node (blob) [fill=black!20!white, rectangle, rounded corners,minimum height=0.5cm,minimum width=2cm, inner sep=0pt] at (0.1,0) {};
\node[anch] (rho) at (0.5,1) {};
\node (u) at (-1,-1) {$\strut u_+$};
\node (v) at (1,-1) {$\strut v_-$};
\node[anch] (rhoin) at (-0.5,0) {};
\coordinate [label=right:{$\strut{a}_-$}] (alpha)  at (0.5,0);
\draw (u) to(rhoin) (v)--(alpha) (alpha) to[out=60,in=0] (rho);
\draw[dotted,->,>=open triangle 60] (rho) to[out=180, in=90] (blob);
\draw[dotted, ->-=0.6] (rhoin)--(alpha);
\end{tikzpicture}
-
\begin{tikzpicture}[thick,scale=1,baseline=-0.6ex]
\node (blob) [fill=black!20!white, rectangle, rounded corners,minimum height=0.5cm,minimum width=2cm, inner sep=0pt] at (0.1,0) {};
\node[anch] (rho) at (-0.2,-1) {};
\node (u) at (-1,-1) {$\strut u_+$};
\node (v) at (1,-1) {$\strut v_-$};
\node[anch] (rhoin) at (-0.5,0) {};
\coordinate [label=right:{$\strut {a}_-$}] (alpha)  at (0.5,0);
\draw (u) to(rho) (v)--(alpha) (alpha) ..controls +(1,1) and +(-1,1)..(rhoin);
\draw[dotted,->,>=open triangle 60] (rho) to[out=0, in=-90] (blob);
\draw[dotted, ->-=0.6] (rhoin)--(alpha);
\end{tikzpicture}
-
\begin{tikzpicture}[thick,scale=1,baseline=-0.6ex]
\node[anch] (rho) at (-0.2,0) {};
\node (u) at (-1,-1) {$\strut u_+$};
\node (v) at (1,-1) {$\strut v_-$};
\coordinate (x) at (-0.7,0) {};
\coordinate [label=right:{$\strut{a}_-$}] (alpha)  at (0.5,0);
\draw (u) to(x) (v)--(alpha) (alpha) ..controls +(1,1) and +(-1,1)..(x)--(rho);
\draw[dotted,->-=0.6] (rho) -- (alpha);
\end{tikzpicture}
\end{multline*}
$$
-\delta_{a}
\begin{tikzpicture}[scale=0.6, baseline=-0.6ex, thick]
\node (u) at (-0.5,0) {$\strut u_+$};
\coordinate (x) at (1,0);
\coordinate (y) at (3,0) {};
\node (v) at (4.5,0) {$\strut v_-$};
\node (plus) at (2,0.8) [proj] {\tiny $+$};
\node (minus) at (2,-0.8) [proj] {\tiny$-$};
\draw (u)--(x) (y)--(v);
\draw (x) to[bend left] (plus) (plus)to[bend left] (y)
      (x) to[bend right] (minus) (minus)to[bend right] (y) ;
\end{tikzpicture}
=
\begin{tikzpicture}[scale=0.6, baseline=-0.6ex, thick]
\node (u) at (-0.5,0) {$\strut u_+$};
\coordinate(x) at (1,0);
\node (alpha) at (0.7,0.4) {${a}_+$};
\coordinate (y) at (3,0) {};
\node (v) at (4.5,0) {$\strut v_-$};
\node (rho) at (2,-0.8) [anch] {};
\draw (u)--(x) (y)--(v);
\draw (x) .. controls +(0.4,1) and +(-0.4,1).. (y)
      (y) to[bend left] node[proj]{\tiny$-$} (rho);
\draw [dotted,->-=0.6](rho)to[bend left] (x);
\end{tikzpicture}
+
\begin{tikzpicture}[scale=0.6, baseline=-0.6ex, thick]
\node (u) at (-0.5,0) {$\strut u_+$};
\coordinate(x) at (1,0);
\coordinate (y) at (3,0) {};
\node (alpha) at (3.4,-0.4) {${a}_-$};
\node (v) at (4.5,0) {$\strut v_-$};
\node (rho) at (2,0.8) [anch] {};
\draw (u)--(x) (y)--(v);
\draw (x) .. controls +(0.4,-1) and +(-0.4,-1).. (y)
      (x) to[bend left] node[proj]{\tiny$+$} (rho);
\draw [dotted,->-=0.6](rho)to[bend left] (y);
\end{tikzpicture}
$$

\begin{multline*}
\delta_{a} \begin{tikzpicture}[scale=0.6, baseline=-0.6ex, thick]
\node (u) at (-1,-1) {$u_+$};
\node (v) at (1,-1) {$v_-$};
\coordinate (x) at (0,0);
\node (minus) at (-0.2,1) [proj] {\tiny$-$};
\draw (u)--(x)--(v) (x)to[out=80, in=-60](minus);
\node (rho) at (-1,0.7) [anch]  {};
\draw (minus) to[out=140,in=90] (rho);
\draw[dotted,->,>=open triangle 60] (rho) to[out=-70,in=160] (x);
\end{tikzpicture}
=
-
\begin{tikzpicture}[thick,scale=1,baseline=-0.6ex]
\node (blob) [fill=black!20!white, rectangle, rounded corners,minimum height=0.5cm,minimum width=2cm, inner sep=0pt] at (0.1,0) {};
\node[anch] (rho) at (0.5,1) {};
\node (u) at (-1,-1) {$ u_+$};
\node (v) at (1,-1) {$ v_-$};
\node[anch] (rhoin) at (-0.5,0) {};
\coordinate [label=right:{$\strut{a}_-$}] (alpha)  at (0.5,0);
\draw (u) to(rhoin) (v)--(alpha) (alpha) to[out=60,in=0] (rho);
\draw[dotted,->,>=open triangle 60] (rho) to[out=180, in=90] (blob);
\draw[dotted, ->-=0.6] (rhoin)--(alpha);
\end{tikzpicture}
+
\begin{tikzpicture}[baseline=-0.6ex, thick, scale=1]
\node (u) at (-0.5,-1) {$ u_+$};
\node (v) at (0.5,-1) {$ v_-$};
\coordinate (x) at (0,-0.3);
\coordinate (alpha) at (0,0.2);
\draw (0.3,0) node {${a}_-$};
\node[anch] (rho) at (0,1) {};
\draw (u)--(x)--(v) (x)--(alpha)..controls +(1,0.5) and +(0.5,0).. (rho);
\draw[dotted, ->-=0.75] (rho) ..controls +(-0.5,0) and +(-1,0.5).. (alpha);
\end{tikzpicture}
\\
+
\begin{tikzpicture}[thick,scale=1,baseline=-0.6ex]
\node[anch] (rho) at (0.5,1) {};
\node (u) at (-1,-1) {$ u_+$};
\node (v) at (1,-1) {$ v_-$};
\coordinate (x) at (-0.5,0) {};
\coordinate  (alpha)  at (0.5,0);
\draw (0.3,-0.3) node {${a}_-$};
\draw (u)--(x)--(alpha)--(v);
\draw (x) ..controls +(-1,1) and +(-0.5,0.3)..node[proj] {\tiny$-$} (rho);
\draw[dotted, ->-=0.6] (rho)to[bend left=50](alpha);
\end{tikzpicture}
-
\begin{tikzpicture}[thick,scale=1,baseline=-0.6ex]
\node[anch] (rho) at (-0.5,1) {};
\node (u) at (-1,-1) {$ u_+$};
\node (v) at (1,-1) {$ v_-$};
\coordinate (x) at (0.5,0) {};
\coordinate (alpha)  at (-0.5,0);
\draw (-0.25,-0.3) node {${a}_+$};
\draw (u)--(alpha)--(x)--(v);
\draw (x) ..controls +(1,1) and +(0.5,0.3)..node[proj] {\tiny$-$} (rho);
\draw[dotted, ->-=0.6] (rho)to[bend right=50](alpha);
\end{tikzpicture}
\end{multline*}

Since
$$
\la[s_-,u_+],v_-\ra=-\la[u_+,s_-],v_-\ra=
\begin{tikzpicture}[baseline=-0.6ex, thick, scale=1]
\node (u) at (-0.5,-1) {$ u_+$};
\node (v) at (0.5,-1) {$ v_-$};
\coordinate (x) at (0,-0.3);
\coordinate (alpha) at (0,0.2);
\draw (0.3,0.05) node {${a}_-$};
\node[anch] (rho) at (0,1) {};
\draw (u)--(x)--(v) (x)--(alpha)..controls +(1,0.5) and +(0.5,0).. (rho);
\draw[dotted, ->-=0.75] (rho) ..controls +(-0.5,0) and +(-1,0.5).. (alpha);
\end{tikzpicture}
-
\begin{tikzpicture}[thick,scale=1,baseline=-0.6ex]
\node (blob) [fill=black!20!white, rectangle, rounded corners,minimum height=0.5cm,minimum width=2cm, inner sep=0pt] at (0.1,0) {};
\node[anch] (rho) at (-0.2,-1) {};
\node (u) at (-1,-1) {$\strut u_+$};
\node (v) at (1,-1) {$\strut v_-$};
\node[anch] (rhoin) at (-0.5,0) {};
\coordinate [label=right:{$\strut{a}_-$}] (alpha)  at (0.5,0);
\draw (u) to(rho) (v)--(alpha) (alpha) ..controls +(1,1) and +(-1,1)..(rhoin);
\draw[dotted,->,>=open triangle 60] (rho) to[out=0, in=-90] (blob);
\draw[dotted, ->-=0.6] (rhoin)--(alpha);
\end{tikzpicture}
$$
we get
$$\delta_{a}\on{GRic}_{V_+}^{(\nabla)}(u_+,v_-)=\la[s_-,u_+],v_-\ra$$
as
$$
-
\begin{tikzpicture}[thick,scale=1,baseline=-0.6ex]
\node[anch] (rho) at (-0.2,0) {};
\node (u) at (-1,-1) {$\strut u_+$};
\node (v) at (1,-1) {$\strut v_-$};
\coordinate (x) at (-0.7,0) {};
\coordinate [label=right:{$\strut{a}_-$}] (alpha)  at (0.5,0);
\draw (u) to(x) (v)--(alpha) (alpha) ..controls +(1,1) and +(-1,1)..(x)--(rho);
\draw[dotted,->-=0.6] (rho) -- (alpha);
\end{tikzpicture}
+
\begin{tikzpicture}[scale=0.6, baseline=-0.6ex, thick]
\node (u) at (-0.5,0) {$\strut u_+$};
\coordinate(x) at (1,0);
\coordinate (y) at (3,0) {};
\node (alpha) at (3.4,-0.4) {${a}_-$};
\node (v) at (4.5,0) {$\strut v_-$};
\node (rho) at (2,0.8) [anch] {};
\draw (u)--(x) (y)--(v);
\draw (x) .. controls +(0.4,-1) and +(-0.4,-1).. (y)
      (x) to[bend left] node[proj]{\tiny$+$} (rho);
\draw [dotted,->-=0.6](rho)to[bend left] (y);
\end{tikzpicture}
+
\begin{tikzpicture}[thick,scale=1,baseline=-0.6ex]
\node[anch] (rho) at (0.5,1) {};
\node (u) at (-1,-1) {$ u_+$};
\node (v) at (1,-1) {$ v_-$};
\coordinate (x) at (-0.5,0) {};
\coordinate  (alpha)  at (0.5,0);
\draw (0.3,-0.3) node {${a}_-$};
\draw (u)--(x)--(alpha)--(v);
\draw (x) ..controls +(-1,1) and +(-0.5,0.3)..node[proj] {\tiny$-$} (rho);
\draw[dotted, ->-=0.6] (rho)to[bend left=50](alpha);
\end{tikzpicture}
=0
$$
and the remaining terms cancel in pairs.

Finally, since $\la[s_-,u_+],v_-\ra$ is independent of $\nabla$, we get
$$\on{GRic}_{V_+}^{(\nabla+{a})}(u_+,v_-)-\on{GRic}_{V_+}^{(\nabla)}(u_+,v_-)=\la[s_-,u_+],v_-\ra$$
as we wanted to show.

\end{document}